\documentclass[12pt]{amsart}

\usepackage{amsmath, enumerate, amsfonts, amssymb, amsthm, graphics, graphicx, verbatim, caption, subcaption, MnSymbol, booktabs, multirow, soul, array, hyperref, float, quiver, adjustbox, comment}
\usepackage[numbers]{natbib}
\usepackage[table]{xcolor}
\usepackage[english]{babel}
\usepackage[margin=0.95in]{geometry}
\usepackage{tikz, float}
\usepackage{mleftright}
\mleftright
\usetikzlibrary{shapes}
\usetikzlibrary{positioning}
\usetikzlibrary{decorations.markings}
\usetikzlibrary{arrows}
\usepackage{quiver}
\tikzstyle{subgroup}=[scale=1]

\newtheorem{thm}{Theorem}
\newtheorem{theorem}{Theorem}[section]
\newtheorem{corollary}[theorem]{Corollary}
\newtheorem{proposition}[theorem]{Proposition}
\newtheorem{lemma}[theorem]{Lemma}

\title[A M\"obius function]{A M\"obius function on the Centralizer Lattice}
\author[Cocke]{William Cocke}
\address{School of Computer and Cyber Sciences, Augusta University, Augusta, GA 30912; \newline \indent
wcocke@augusta.edu \\ \newline
\indent And Carnegie Mellon University, Pittsburgh, PA 15213: wcocke@andrew.cmu.edu}
\author[Lewis]{Mark L. Lewis}
\address{Department of Mathematical Sciences, Kent State University, Kent, OH  44242; \newline \indent lewis@math.kent.edu}
\author[McCulloch]{Ryan McCulloch}
\address{Department of Mathematics and Statistics, Binghamton University, Binghamton, NY 13902; \newline \indent rmccullo1985@gmail.com}
\date{}

\begin{document}

\begin{abstract}
We consider the M\"obius function on the poset of element centers and obtain some new results regarding centralizers in a $p$-group.
\end{abstract}



\maketitle

\section{Introduction}

In this paper, all groups are finite.  In the study of finite groups, M\"obius functions play an important role in understanding the lattice of subgroups.  So far as we can tell, a M\"obius function on the lattice of subgroups was first defined in 1936 by Phillip Hall in \cite{hall}.  In \cite{gaschutz}, Gasch\"utz implicitly computed the values for the M\"obius function on the lattice of subgroups, and using this work, Kratzer and Th\'evanez in \cite{KrTh} present an explicit version of this formula. 

The paper by Hawkes, Isaacs, and \"Ozaydin \cite{HIO} written in 1989 presents a uniform and elementary treatment of the theory of M\"obius functions on the lattice of subgroups of a group.  We recommend this well-written paper to the reader who wants to get a full understanding of the state of M\"obius functions and groups at that time.

Since \cite{HIO} was written, M\"obius functions on groups have continued to be studied, and we will not try to list all of the papers that have been written.  We will mention just a few.  For example, in \cite{pahlings}, Pahlings shows that the equation that Hawkes, Isaacs, and \"Ozaydin found for the M\"obius function on solvable groups is not always valid for nonsolvable groups.  Pahlings also looks at the connection between the M\"obius function on the lattice of subgroups and the M\"obius function on the poset of conjugacy classes of subgroups.  This relationship is further studied by Dalla Volta and Zini in \cite{DaZi}.  In the recent paper of Dalla Volta and Lucchini \cite{DaLu}, they consider a M\"obius function on the poset of automorphism classes of subgroups of a group. 

In this paper, we define a M\"obius function on the poset of element centers (a subposet of the full centralizer lattice) to obtain results on centralizers in $p$-groups.  

\begin{thm} \label{intro: Mob}
Let $G$ be a nonabelian $p$-group.  Let $\mu$ be the M\"{o}bius function defined on the poset $\{ {\bf Z}(\mathbf{C}_G(g)) \mid g \in G \}$ under containment.  For each $H \in \mathfrak{C}(G) \setminus \{G\}$, where $\mathfrak{C}(G)$ denotes the centralizer lattice of $G$, we have 
$$\sum_{\substack {H \subseteq \mathbf{C}_G(g), \\ \mathbf{C}_G(g) \subset G}} \mu(\mathbf{Z}(\mathbf{C}_G(g))) \equiv -1 \bmod p.$$
\end{thm}

In \cite{centgraph}, the second and third authors have introduced a graph on the set of centralizers of noncentral elements.  We define the {\it centralizer graph}, $\Gamma_{\mathcal{Z}} (G)$, of $G$ to be the graph whose vertex set is $\mathcal{C} (G) = \{ \mathbf{C}_G (g) \mid g \in G \setminus \mathbf{Z}(G) \}$ and there is an edge between distinct $\mathbf{C}_G (g)$ and $\mathbf{C}_G (h)$ if $\mathbf{Z}(\mathbf{C}_G (h)) \subseteq \mathbf{C}_G (g)$.  In \cite{centgraph}, we show the centralizer graph of $G$ is closely related to the commuting graph which we will define later.  An {\it $F$-group} is a group $G$ where if $\mathbf{C}_G (g), \mathbf{C}_G(h)$ are distinct members of $\mathcal{C} (G)$, we have $\mathbf{C}_G (g)$ does not contain $\mathbf{C}_G (h)$ and $\mathbf{C}_G (h)$ does not contain $\mathbf{C}_G (g)$.  For $F$-groups that are $p$-groups, we obtain the following characterization of their centralizer graphs.

\begin{thm}\label{intro: cent_graph_degree}
Suppose $G$ is a nonabelian $F$-group that is a $p$-group.  Every vertex of the centralizer graph $\Gamma_{\mathcal{Z}}(G)$ has degree congruent to $0$ modulo $p$.
\end{thm}

We have a number of examples of $p$-groups that are not $F$-groups where the vertices of the centralizer graph do not have degrees that are all congruent to the same value modulo $p$.

The paper proceeds as follows. In Section \ref{sec: basic} we recall some of the basic properties of the centralizer. Some of these properties are less well-publicized, and the section contains proofs of them.  Section \ref{sec: op} contains our material on $\mathbf{C}_G(\cdot)$ as an operator and on the centralizer lattice.  In Sections \ref{sec: elems} and \ref{sec: Moby} we define the equivalence relation and present our $p$-group results using the M\"obius function.

\section{Basic properties of centralizers}\label{sec: basic}

In this section, we give some background on centralizers.  We then state and prove some of the well-known properties of centralizers.  We use the bracket notation for the `subgroup-generated-by' and the union symbol for the set-theoretic union. Since our goal is to investigate properties of the centralizer as a map inside a fixed group, i.e., $\mathbf{C}_G(\cdot)$ for a group $G$, we explore how the centralizer behaves with respect to collections of subgroups of $G$.

The centralizer of a subset is a standard topic in introductory group theory courses, and it has also enjoyed great study in the literature.  The study of the lattice of all centralizers dates back to Schmidt, see \cite{sch70}, \cite{sch_book}.  Properties of centralizers of sets have been studied in, for example, \cite{bry}, \cite{dun}, and \cite{tre}.  Researchers also consider how centralizers of sets affect the structure of the group, see for example \cite{coc}, \cite{fal}, and \cite{del}.  Sections \ref{sec: basic} \& \ref{sec: op} of this manuscript outline the properties of the centralizer lattice of a group and the centralizer operator on subgroups of a group.  Many of these results are ``mathematical folklore'', and thus known to experts, and may extend beyond groups to non-commutative associative algebras.  Our treatment is clear and unified, restricted to the class of all groups, and may be a valuable reference for researchers in group theory.

In many areas of the literature, the term centralizer is used to refer to the centralizer of an element in a group (for us we will make the distinction between the centralizer of a set and centralizer of one element).  The literature is full of results on centralizers of elements, for example: numbers of centralizers (\cite{kho_21}), coverings by centralizers (\cite{HaAm}, \cite{cover}), two elements having the same centralizer (\cite{rah_25}), groups with centralizers of the same size (\cite{DHJ}).  In Section \ref{sec: elems}, we consider the element centralizers and their centers, and the role these centralizers play in the centralizer lattice.

Given a subset $S$ of a group $G$, define the {\it centralizer} of $S$ in $G$, written $\mathbf{C}_G(S)$ as the set $\{ x \in G \mid xs = sx \,\, \forall s \in S \}$. A standard exercise is to show that the set $\mathbf{C}_G(S)$ is actually a subgroup for any subset $S$ of $G$. Another standard operator involving subsets of a group is the `subset-generated-by' operator $\langle \cdot \rangle$ where $\langle S \rangle$ is the subgroup of $G$ generated by $S$.

Given a subset $S$ of a group $G$, define $\mathbf{C}_G(S) = \{ x \in G \,\, | \,\, xs = sx \,\, \forall s \in S \}$.  Note that $\mathbf{C}_G (\emptyset) = G$. When $S = \{ s \}$, we write $\mathbf{C}_G(\{ s \}) = \mathbf{C}_G(s)$.  We call centralizers of the form $\mathbf{C}_G(s)$ \textit{element centralizers}, a slightly shortened version of centralizers of elements.

\begin{lemma} \label{lem: contain}
Let $G$ be a group and $S,T$ subsets of $G$.  We have $$S \subseteq T \implies \mathbf{C}_G(T) \subseteq \mathbf{C}_G(S).$$
\end{lemma}

\begin{proof}
Suppose $S \subseteq T$.  We have
$$x \in \mathbf{C}_G(T) \iff xg = gx \,\, \forall g \in T \implies xg = gx \,\, \forall g \in S \iff x \in \mathbf{C}_G(S).$$
\end{proof}

\begin{proposition} \label{prop: cent_union_int_1}
Let $G$ be a group and let $\mathfrak{S}$ be a nonempty collection of subsets of $G$.  
\begin{enumerate}
\item $\displaystyle \mathbf{C}_G (\bigcup_{S \in \mathfrak{S}} S) = \bigcap_{S \in \mathfrak{S}} \mathbf{C}_G(S).$
\item $\displaystyle \bigcup_{S \in \mathfrak{S}} \mathbf{C}_G(S)\subseteq \mathbf{C}_G (\bigcap_{S \in \mathfrak{S}} S).$
\end{enumerate}
\end{proposition}

\begin{proof}
We have
\begin{align*}
x \in \mathbf{C}_G (\bigcup_{S \in \mathfrak{S}} S) \iff \\
xg = gx \,\,  \forall g \in \bigcup_{S \in \mathfrak{S}} S \iff \\
\forall \,\, S \in \mathfrak{S}, xg = gx \,\,  \forall g \in S \iff \\
\forall \,\, S \in \mathfrak{S}, x \in \mathbf{C}_G(S) \iff \\
x \in \bigcap_{S \in \mathfrak{S}} \mathbf{C}_G(S).
\end{align*}

We have
\begin{flalign*}
 x \in \bigcup_{S \in \mathfrak{S}} \mathbf{C}_G(S) \iff \\
\exists \,\, S' \in \mathfrak{S}, x \in \mathbf{C}_G(S') \implies \\
 x \in \mathbf{C}_G (\bigcap_{S \in \mathfrak{S}} S),
\end{flalign*}
since $\displaystyle \bigcap_{S \in \mathfrak{S}} S \subseteq S'$ for each $S' \in \mathfrak{S}$, and Lemma \ref{lem: contain} applies.
\end{proof}

\begin{lemma} \label{lem: sub}
Let $G$ be a group and let $S$ be a subset of $G$.  We have $\mathbf{C}_G(S)$ is a subgroup of $G$ and we have $\mathbf{C}_G(S) = \mathbf{C}_G(\langle S \rangle)$.
\end{lemma}

\begin{proof}
Note that $1 s = s = s 1$ for all $s \in S$, and so $1 \in \mathbf{C}_G(S)$.  Let $x,y \in \mathbf{C}_G(S)$.  Since $y  s = s  y$ for all $s \in S$, by multiplying on the right and on the left by $y^{-1}$, we obtain $s  y^{-1} = y^{-1}  s$ for all $s \in S$.  Thus $y^{-1} \in \mathbf{C}_G(S)$.  Now, $x  y^{-1}  s = x  s  y^{-1} = s  x  y^{-1}$ for all $s \in S$.  Hence, $x y^{-1} \in \mathbf{C}_G(S)$ and it follows that $\mathbf{C}_G(S)$ is a subgroup of $G$.

Since $S \subseteq \langle S \rangle $, it follows by Lemma \ref{lem: contain} that $\mathbf{C}_G(\langle S \rangle) \subseteq \mathbf{C}_G(S)$.  If $S=\emptyset$, then $G = \mathbf{C}_G(\{ 1 \}) = \mathbf{C}_G(\langle \emptyset \rangle) = \mathbf{C}_G(\emptyset)$, so suppose $S$ is nonempty.  Note that if $gx = xg$ for elements $g,x \in G$, then multiplying on the right and on the left by $x^{-1}$, we obtain $x^{-1}g = g x^{-1}$.  Furthermore, note that if $gx = xg$ and $gy = yg$ for elements $g,x,y \in G$, then we have $gxy = xgy = xyg$.  It follows inductively that $$g \in \mathbf{C}_G(S) \implies g \in \mathbf{C}_G(\langle S \rangle),$$ since an element of $\langle S \rangle$ is a product consisting of elements of $S$ and inverses of elements of $S$.  Hence we have $\mathbf{C}_G(S) = \mathbf{C}_G(\langle S \rangle)$.
\end{proof}
 
\begin{proposition} \label{prop: cent_union_int}
Let $G$ be a group and let $\mathfrak{S}$ be a nonempty collection of subsets of $G$.
\begin{enumerate}
\item $\displaystyle \mathbf{C}_G (\bigcup_{S \in \mathfrak{S}} S) = \mathbf{C}_G (\langle \bigcup_{S \in \mathfrak{S}} S \rangle) = \bigcap_{S \in \mathfrak{S}} \mathbf{C}_G(S).$
\item $\displaystyle \bigcup_{S \in \mathfrak{S}} \mathbf{C}_G(S)\subseteq \langle \bigcup_{S \in \mathfrak{S}} \mathbf{C}_G(S) \rangle \subseteq \mathbf{C}_G (\bigcap_{S \in \mathfrak{S}} S).$
\end{enumerate}
\end{proposition}

\begin{proof}
Much of this has been proven in Proposition \ref{prop: cent_union_int_1}.

It follows by Lemma \ref{lem: sub} that $\displaystyle \mathbf{C}_G (\bigcup_{S \in \mathfrak{S}} S) = \mathbf{C}_G (\langle \bigcup_{S \in \mathfrak{S}} S \rangle)$.

Since $\displaystyle \mathbf{C}_G (\bigcap_{S \in \mathfrak{S}} S)$ is a subgroup of $G$ that contains the set $\displaystyle \bigcup_{S \in \mathfrak{S}} \mathbf{C}_G(S)$, we have $\displaystyle \langle \bigcup_{S \in \mathfrak{S}} \mathbf{C}_G(S) \rangle \subseteq \mathbf{C}_G (\bigcap_{S \in \mathfrak{S}} S)$.
\end{proof}

We give an example having strict containments in Proposition \ref{prop: cent_union_int} (2).  Let $G = S_4$ be the symmetric group on $\{1,2,3,4\}$, and let $$\mathfrak{S} = \left\{ \left\{ (1,2,3) \right\}, \left\{ (1,2)(3,4),(1,3)(2,4) \right\} \right\}.$$  We have \begin{flalign*}
& \bigcap_{S\in \mathfrak{S}} =  \left\{ (1,2,3) \right\} \bigcap  \left\{ (1,2)(3,4),(1,3)(2,4) \right\} = \emptyset, \\ 
&\mathbf{C}_G((1,2,3))  = \left\langle (1,2,3) \right\rangle, \\ \text{ and } \\
&\mathbf{C}_G(\{ (1,2)(3,4), (1,3)(2,4) \}) = \left\langle (1,2)(3,4), (1,3)(2,4) \right\rangle.\end{flalign*} 

Observe 
$$\underbrace{\left\langle (1,2,3) \right\rangle}_{\mathbf{C}_G\left((1,2,3)\right)} \bigcup \underbrace{\left\langle (1,2)(3,4),(1,3)(2,4) \right\rangle}_{\mathbf{C}_G\left(\left\{(1,2)(3,4),(1,3)(2,4)\right\}\right)}\subset \underbrace{\left\langle \left\langle (1,2,3) \right\rangle \cup \left\langle (1,2)(3,4),(1,3)(2,4) \right\rangle \right\rangle }_{A_4} \subset \underbrace{\mathbf{C}_G(\emptyset)}_{G}.$$

Thus, the union of centralizers need not be a subgroup, nor generate a centralizer. In the next section we will discuss how the join operation should be defined on two centralizers to ensure that the output is also a centralizer. Before we can do so, we will discuss the operation $\mathbf{C}_G(\mathbf{C}_G(\cdot))$ in the context of Galois connections.

\section{The map \texorpdfstring{$\mathbf{C}_G (\cdot)$}{C} and the centralizer lattice}\label{sec: op}

In this section, we explore the centralizer operator $\mathbf{C}_G(\cdot)$ on $\mathcal{P}(G)$ the power set of a group $G$. One of our main observations is that $\mathbf{C}_G(\mathbf{C}_G(\cdot))$ is a \emph{closure operator} on $\mathcal{P}(G)$.  We compare this closure operator to the well-known closure operator $\langle \cdot \rangle$. Just as the subgroup lattice represents the fixed points of $\langle \cdot \rangle$, the centralizer lattice consists of the fixed points of $\mathbf{C}_G(\mathbf{C}_G(\cdot))$.  We state our observations in the following theorem.

\begin{thm}
Let $G$ be a group and let $\mathfrak{C} (G) = \{ {\bf C}_G (H) \mid H \in \mathcal{P} (G) \}$, then the following are true:
\begin{enumerate}
\item $\mathfrak{C} (G)$ is a lattice.
\item The map ${\bf C}_G (\cdot): \mathfrak {C} (G) \rightarrow \mathfrak {C} (G)$ is an inclusion reversing bijection.
\item ${\bf C}_G ({\bf C}_G (H)) = H$ for all $H \in \mathfrak{C} (G)$.
\end{enumerate}
\end{thm}

We pause briefly to consider Galois connections on posets. This is a beautiful area of mathematics and perhaps a gateway math to category theory or universal algebra. 

Given posets $\mathcal{A}$ and $\mathcal{B}$, a \textit{Galois connection} is a pair $(f,g)$ of order-reversing maps $f : \mathcal{A} \rightarrow \mathcal{B}$ and $g : \mathcal{B} \rightarrow \mathcal{A}$ such that $$Y \leq f(X) \iff X \leq g(Y)$$ for every $X \in \mathcal{A}$ and for every $Y \in \mathcal{B}$.

The definition we give here is the classical Galois connection where the maps are order-reversing.  It can be shown that in any Galois connection, the compositions $g \circ f : \mathcal{A} \rightarrow \mathcal{A}$ and $f \circ g : \mathcal{B} \rightarrow \mathcal{B}$ are closure operators.

A {\it closure operator}, $\mathrm{cl}$, is a map from a poset $\mathcal{P}$ to itself that satisfies the following three axioms for every $X,Y \in \mathcal{P}$.
\begin{enumerate}
\item (Extensive) $X \leq \mathrm{cl}(X)$.
\item (Increasing) $X \leq Y \implies \mathrm{cl}(X) \leq \mathrm{cl}(Y)$.
\item (Idempotent) $\mathrm{cl}(\mathrm{cl}(X)) = \mathrm{cl}(X)$.
\end{enumerate}

The reader is familiar with a closure operator on the power set of a group $\mathcal{P}(G)$. The subgroup generated by operator, $\langle \cdot \rangle$, is a closure operator on $\mathcal{P}(G)$. After showing that $\mathbf{C}_G(\mathbf{C}_G(\cdot))$ is a closure operator, we will explore ways in which its behavior mimics that of $\langle \cdot \rangle.$ 

\begin{proposition} \label{prop: Galois}
Let $G$ be a group and let $\mathcal{P}(G)$ be the power set of $G$ poset under containment.  The pair $(\mathbf{C}_G (\cdot), \mathbf{C}_G (\cdot))$ is an order-reversing Galois connection between $\mathcal{P}(G)$ and itself.
\end{proposition}

\begin{proof}
Lemma \ref{lem: contain} shows that $\mathbf{C}_G(\cdot)$ is order reversing.   And for $S, T$ subsets of $G$, 
$$T \subseteq \mathbf{C}_G(S) \iff S \subseteq \mathbf{C}_G(T)$$ 
holds since both of these say that every element of $S$ commutes with every element of $T$.
\end{proof}

It follows that $\mathbf{C}_G(\mathbf{C}_G(\cdot))$ is a closure operator on the power set of a group $G$. We illustrate this in Figure \ref{fig: cc_d8} for the subgroups of the group $G=D_8$.  
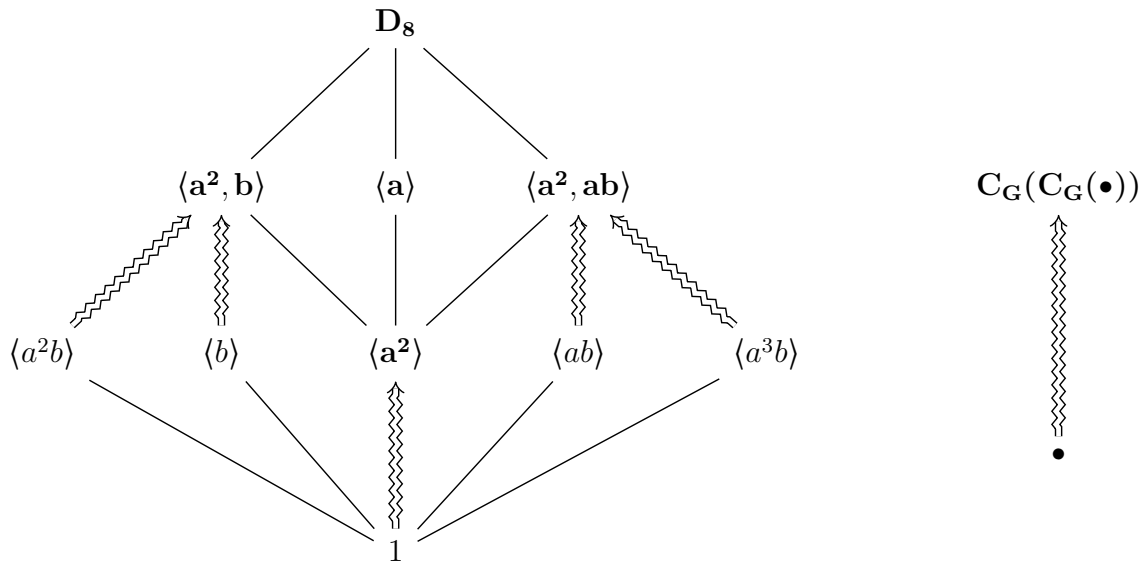
\begin{figure}
\[\begin{tikzcd}
	&& { \mathbf{D_8}} \\
	\\
	& { \mathbf{ \langle a^2,b \rangle}} & { \mathbf{ \langle a\rangle }} & { \mathbf{\langle a^2,ab\rangle}} &&& { \mathbf{C_G(C_G(\bullet))}} \\
	\\
	{ \langle a^2b \rangle} & { \langle b \rangle} & { \mathbf{\langle a^2\rangle}} & { \langle ab\rangle} & {\langle a^3 b\rangle} \\
	&&&&&& \bullet \\
	&& 1 \\
	\arrow[no head, from=3-2, to=1-3]
	\arrow[no head, from=3-2, to=5-3]
	\arrow[no head, from=3-3, to=1-3]
	\arrow[no head, from=3-4, to=1-3]
	\arrow[Rightarrow, squiggly, from=5-1, to=3-2]
	\arrow[no head, from=5-1, to=7-3]
	\arrow[Rightarrow, squiggly, from=5-2, to=3-2]
	\arrow[no head, from=5-2, to=7-3]
	\arrow[no head, from=5-3, to=3-3]
	\arrow[no head, from=5-3, to=3-4]
	\arrow[Rightarrow, squiggly, from=5-4, to=3-4]
	\arrow[no head, from=5-4, to=7-3]
	\arrow[Rightarrow, squiggly, from=5-5, to=3-4]
	\arrow[no head, from=5-5, to=7-3]
	\arrow[Rightarrow, squiggly, from=6-7, to=3-7]
	\arrow[Rightarrow, squiggly, from=7-3, to=5-3]
\end{tikzcd}\]
\caption{A visualization of the map $\mathbf{C}_G(\mathbf{C}_G(\cdot))$ for $G=D_8$. For a subgroup in $D_8$, follow the double arrows until you arrive at a bolded subgroup, which is a fixed point of the closure operator. Note that this group, $G=D_8$ has the property that you only have to move once---in general this is not true.} \label{fig: cc_d8}
\end{figure}

\begin{proposition}\label{prop: triple_g}
Let $\mathcal{P}$ be a poset and suppose $g:\mathcal{P}\rightarrow \mathcal{P}$ is an order-reversing map such that the pair $(g,g)$ is a Galois connection between $\mathcal{P}$ and itself. Then for any $S \in \mathcal{P}$ we have $g(g(g(S)))=g(S)$. 
\end{proposition}

\begin{proof}
Since $g\circ g$ is a closure operation, it is extensive, so $S\subseteq g(g(S)).$ Since $g$ is order-reversing, we also have $$g(g(g(S))) \subseteq g(S).$$ Applying extensivity to $g(S)$ we have $$g(S) \subseteq g(g(g(S))).$$ Therefore $g(S) = g(g(g(S))).$  
\end{proof}

We state the corollary for $\mathbf{C}_G(\mathbf{C}_G(\cdot))$.

\begin{corollary}\label{cor: triple_C}
Let $G$ be a group and $S \subseteq G$. Then $\mathbf{C}_G(\mathbf{C}_G(\mathbf{C}_G(S))) = \mathbf{C}_G(S)$.
\end{corollary}

\begin{proof}
By Proposition \ref{prop: Galois}, $(\mathbf{C}_G (\cdot),\mathbf{C}_G (\cdot))$ is an order-reversing Galois connection between $\mathcal{P}(G)$ and itself, and Proposition \ref{prop: triple_g} applies.
\end{proof}

Let $\mathfrak{C}(G)$ denote the image of the map $\mathbf{C}_G(\cdot) : \mathcal{P}(G) \rightarrow \mathcal{P}(G)$.  Hence $$\mathfrak{C}(G) = \{ \mathbf{C}_G(S) \mid S \subseteq G \}$$ is the set of all centralizers in $G$.  Observe the following.

\begin{proposition}\label{prop: cent_basic}
Let $G$ be a group and $S,T$ be subsets of $G$.  
\begin{enumerate}
\item We have $S = \mathbf{C}_G(\mathbf{C}_G(S))$ if and only if $S \in \mathfrak{C}(G)$.
\item We have $S \subseteq T$ implies $\mathbf{C}_G(T) \subseteq \mathbf{C}_G(S)$.
\item If $T \in \mathfrak{C}(G)$, then $S \subseteq T$ if and only if $\mathbf{C}_G(T) \subseteq \mathbf{C}_G(S)$.
\end{enumerate}
\end{proposition}

\begin{proof}
For item (1), clearly $S = \mathbf{C}_G(\mathbf{C}_G(S))$ implies that $S \in \mathfrak{C}(G)$.  Conversely, suppose $S \in \mathfrak{C}(G)$.  So $S = \mathbf{C}_G(T)$ for some $T \subseteq G$.  By Corollary \ref{cor: triple_C}, $$S = \mathbf{C}_G(T) = \mathbf{C}_G(\mathbf{C}_G(\mathbf{C}_G(T))) = \mathbf{C}_G(\mathbf{C}_G(S)).$$

Item (2) is Lemma \ref{lem: contain}.  For item (3), suppose $T \in \mathfrak{C}(G)$ and let $\mathbf{C}_G(T) \subseteq \mathbf{C}_G(S)$.  Since $T \in \mathfrak{C}(G)$, we have $T = \mathbf{C}_G(X)$ for some $X \subseteq G$. So $\mathbf{C}_G(\mathbf{C}_G(X)) \subseteq \mathbf{C}_G(S)$.  Applying Lemma \ref{lem: contain} again, we obtain $$\mathbf{C}_G(\mathbf{C}_G(S)) \subseteq \mathbf{C}_G(\mathbf{C}_G(\mathbf{C}_G(X))).$$

Since $\mathbf{C}_G(\mathbf{C}_G(\cdot))$ is extensive, and by Corollary \ref{cor: triple_C}, we have
$$S \subseteq \mathbf{C}_G(\mathbf{C}_G(S)) \subseteq \mathbf{C}_G(\mathbf{C}_G(\mathbf{C}_G(X))) = \mathbf{C}_G(X) = T.$$
\end{proof}

The following is an immediate corollary.

\begin{corollary} \label{cor: bij}
Let $G$ be a group and $\mathfrak{C}(G)$ be the set of all centralizers of $G$.  The map $\mathbf{C}_G(\cdot) : \mathfrak{C}(G) \rightarrow \mathfrak{C}(G)$ is an order-reversing bijection with inverse given by $\mathbf{C}_G(\cdot)$.
\end{corollary}

The subgroups in $\mathfrak{C}(G)$ form a lattice, see \cite{sch70, sch_book}, where for $H,K \in \mathfrak{C}(G)$, we have meet given by $H \wedge K = H \cap K$ and join given by $H \vee K = \mathbf{C}_G(\mathbf{C}_G(H) \cap \mathbf{C}_G(K))$.  The definition of join given above is equivalent to the definition that $H \vee K$ is the intersection of all of the centralizers containing both $H$ and $K$.  The join $H\vee K$ in the centralizer lattice can properly contain $\langle H , K \rangle $: similarly to what we exhibit after Proposition \ref{prop: cent_union_int}, in $G = S_4$ one finds $\langle H,K \rangle = A_4 \subset G = H \vee K$.

Continuing with our comparison to the $\langle \cdot \rangle$ operator, we note that a given centralizer $H \in \mathfrak{C}(G)$ can be obtained as the centralizers of different subsets. That is $H = \mathbf{C}_G(S) = \mathbf{C}_G(T)$ for different subsets $S,T\subseteq G$.  For example, if $G$ is an abelian group, then $G = \mathbf{C}_G(S)$ for any subset $S$ of $G$.  

Treating $\mathbf{C}_G(\cdot)$ as a map from the power set of a group $G$ to itself, and given a centralizer $\mathbf{C}_G(S)$ in the image, we let $\mathfrak{F}_{\mathbf{C}_G(S)}$ denote the \textit{fiber} of $\mathbf{C}_G(S)$, and so $$\mathfrak{F}_{\mathbf{C}_G(S)} = \{ T \subseteq G \,\, | \,\, \mathbf{C}_G(T) = \mathbf{C}_G(S) \}.$$

Here we can apply some structure to the fibers of the map itself. 

\begin{proposition}\label{prop: unique_cent}
Let $G$ be a group and $S \subseteq G$.  Let $$ \mathfrak{F} = \mathfrak{F}_{\mathbf{C}_G(S)} = \{ T \subseteq G \,\, | \,\, \mathbf{C}_G(T) = \mathbf{C}_G(S) \}.$$  We have $\mathfrak{F}$ is closed under arbitrary unions and $\displaystyle \bigcup_{T \in \mathfrak{F}} T = \mathbf{C}_G(\mathbf{C}_G(S))$.  
\end{proposition}

\begin{proof}
Let $A = \mathbf{C}_G(\mathbf{C}_G(S))$. By Corollary \ref{cor: triple_C}, we have $$\mathbf{C}_G(S) = \mathbf{C}_G(\mathbf{C}_G(\mathbf{C}_G(S))),$$ and so $A$ has the property that $\mathbf{C}_G(A) = \mathbf{C}_G(S)$, i.e. $A \in \mathfrak{F}$.  Hence $$A = \mathbf{C}_G(\mathbf{C}_G(S)) \subseteq \bigcup_{T \in \mathfrak{F}} T.$$ 

Now suppose that $\mathfrak{W} \subseteq \mathfrak{F}$ is arbitrary and let $\displaystyle U = \bigcup_{W \in \mathfrak{W}} W$.  By Proposition \ref{prop: cent_union_int_1}, $$ \mathbf{C}_G(U) = \mathbf{C}_G(\bigcup_{W \in \mathfrak{W}} W) =\bigcap_{W \in \mathfrak{W}} \mathbf{C}_G(W) = \mathbf{C}_G(S),$$ and so $U \in \mathfrak{F}$, and $\mathfrak{F}$ is closed under arbitrary unions.  Since $\mathbf{C}_G\mathbf{C}_G$ is extensive, $$U \subseteq \mathbf{C}_G(\mathbf{C}_G(U)) = \mathbf{C}_G(\mathbf{C}_G(S)) = A.$$ Taking $\mathfrak{W} = \mathfrak{F}$, we obtain $\displaystyle \bigcup_{T \in \mathfrak{F}} T \subseteq A$, and thus $\displaystyle \bigcup_{T \in \mathfrak{F}} T = \mathbf{C}_G(\mathbf{C}_G(S))$.
\end{proof}

Note that an analog of Proposition \ref{prop: unique_cent} is well-known for the closure operator $\langle \cdot \rangle$. The union of the subsets that generate the same subgroup is the subgroup itself. Here, the union corresponds to $\mathbf{C}_G(\mathbf{C}_G(\cdot))$. 

An advantage of writing centralizers as centralizers of centralizers is that the join in the centralizer lattice behaves nicely. Continuing with our analogy of `subgroup-generated-by', given two subgroups $\langle X\rangle $ and $\langle Y \rangle$, their join is $\langle X,Y\rangle$. We gain similar `ease of notation' when writing centralizers as centralizers of underlying objects. 

\begin{proposition}\label{prop:join_nice}
Let $G$ be a group and let $H,K \in \mathfrak{C}(G)$, and write $H = \mathbf{C}_G(A)$ and $K = \mathbf{C}_G(B)$ for $A,B \in \mathfrak{C}(G)$.  Then $H \vee K = \mathbf{C}_G(A \cap B)$.
\end{proposition}

\begin{proof}
We have
    $$H \vee K = \mathbf{C}_G(\mathbf{C}_G(H) \cap \mathbf{C}_G(K)) = \mathbf{C}_G(\mathbf{C}_G(\mathbf{C}_G(A)) \cap \mathbf{C}_G(\mathbf{C}_G(B))) = \mathbf{C}_G(A \cap B),$$ since $A,B \in \mathfrak{C}(G)$. 
\end{proof}

In summary, the closure operator $\mathbf{C}_G(\mathbf{C}_G(\cdot))$ closes a set $S$ with respect to being a centralizer.  Every centralizer $\mathbf{C}_G(S)$ can be rewritten as $\mathbf{C}_G(\mathbf{C}_G(\mathbf{C}_G(S)))$, which means rewriting a centralizer as a centralizer of a centralizer.  Propositions \ref{prop: unique_cent} and \ref{prop:join_nice} show nice properties that the centralizer closure yields for the centralizer map and the centralizer lattice.

Throughout, we have encountered  similarities with the operator $\langle \cdot \rangle$ and $\mathbf{C}_G(\mathbf{C}_G(\cdot))$. We extend this analogy even more in the next section. Recall that in general, when writing a subgroup as $\langle X \rangle$ inside a group $G$, one tries to limit the size of $X$ by removing redundant elements. We show how to do this for the centralizer operator in the next section. 

\section{An equivalence relation via element centralizers}\label{sec: elems}

Continuing with the comparison of $\mathbf{C}_G(\mathbf{C}_G(\cdot))$ to $\langle \cdot \rangle$ we investigate the smallest generating sets with respect to $\mathbf{C}_G(\mathbf{C}_G(\cdot))$. Here, we define a standard generating set for every centralizer subgroup using the centralizers of elements (the so-called element centralizers). Our results here mimic the so-called \emph{cogenerator} relationship for group elements that generate the same cyclic subgroup. 

We obtain results related to containment of element centralizers and, dually, their centers in the centralizer lattice.  We define a `centralizer equivalence relation' on the elements of a group, and we show that a centralizer is the disjoint union of the equivalence classes of the element centers that it contains. 

\begin{thm}\label{intro: z_stars_fin}
Let $G$ be a finite nonabelian group and let $H \in \mathfrak{C}(G) \setminus \{ \mathbf{Z}(G) \}$. There exist elements $g_1, \dots, g_t$ in $G \setminus \mathbf{Z}(G)$ so that $\mathbf{C}_G(g_1), \dots, \mathbf{C}_G(g_t)$ are distinct and comprise all of the proper element centralizers that contain $\mathbf{C}_G(H)$.  Furthermore, there exist subsets $\mathbf{Z}^* (g_i) \subseteq \mathbf{Z}({\bf C}_G (g_i))$ for each $i$ so that $\displaystyle H = (\bigcup_{i=1}^t \mathbf{Z}^*(g_i) )\cup \mathbf{Z}(G)$, and this union is disjoint.
\end{thm}

For a group $G$, define an equivalence relation on the elements by $x \sim y$ if and only if $\mathbf{C}_G(\mathbf{C}_G(x)) = \mathbf{C}_G(\mathbf{C}_G(y))$. We note that $x \sim y$ if and only if $\mathbf{C}_G(x) = \mathbf{C}_G(y)$.  (Note that this equivalence relation is defined in \cite{cover}.)

\begin{proposition}\label{prop: equiv_C-CC}
Let $G$ be a group and $x,y\in G$. Then $\mathbf{C}_G(x) = \mathbf{C}_G(y)$ if and only if $\mathbf{C}_G(\mathbf{C}_G(x)) = \mathbf{C}_G(\mathbf{C}_G(y))$. 
\end{proposition}

\begin{proof}
The forward direction follows by applying $\mathbf{C}_G(\cdot)$ to both sides.  For the reverse direction apply $\mathbf{C}_G(\cdot)$ to both sides and reduce using $\mathbf{C}_G(\mathbf{C}_G(\mathbf{C}_G(\cdot))) = \mathbf{C}_G(\cdot)$. 
\end{proof}

We fix a group $G$, and let $X$ denote a fixed set of representatives for the equivalence classes of $\sim$. (For infinite groups, one may invoke the axiom of choice to define $X$.)  Given a centralizer $H \in \mathfrak{C}(G)$, set $\mathfrak{F}_H^* = \{ T \subseteq X \mid \mathbf{C}_G(T) = H \}$.  Note that $\mathfrak{F}_H^*$ is analogous to the fiber $\mathfrak{F}_H$ defined in the previous section, where $\mathfrak{F}^*$ is restricted to within the representative set $X$.

We now show how to reduce a centralizer of a centralizer to a centralizer of a smaller set that retains nice properties.  Given a centralizer $H \in \mathfrak{C}(G)$, let $U_H^* = \{ x \in X \mid H \subseteq \mathbf{C}_G(x) \}$. The $U$ stands for ``upper''. 

\begin{proposition} \label{prop: upper}
Let $G$ be a group and $H \in \mathfrak{C}(G)$, and consider $U_H^*$ with respect to a fixed representative set $X$ with respect to $\sim$.  Then $H = \mathbf{C}_G(U_H^*)$.  
\end{proposition}

\begin{proof}
Since $H \in \mathfrak{C}(G)$, we have 
$$H = \mathbf{C}_G(\mathbf{C}_G(H)) = \bigcap_{x \in \mathbf{C}_G(H)} \mathbf{C}_G(x).$$
Note that 
$$x \in \mathbf{C}_G(H) \iff H = \mathbf{C}_G(\mathbf{C}_G(H)) \subseteq \mathbf{C}_G(x).$$
Thus, $$H = \bigcap_{H \subseteq \mathbf{C}_G(x)} \mathbf{C}_G(x).$$

Because $X$ is a set of representatives under the relation $\sim$, we have 
$$H = \bigcap_{\substack{H \subseteq \mathbf{C}_G(x), \\ x \in X}} \mathbf{C}_G(x) = \bigcap_{x \in U_H^*} \mathbf{C}_G(x) = \mathbf{C}_G(U_H^*).$$
\end{proof}

The next proposition shows that $\mathfrak{F}_H^*$ and $U_H^*$ yield the same nice properties that we saw in Propositions \ref{prop: unique_cent} and \ref{prop:join_nice}.

\begin{proposition}\label{prop: stars}
Let $G$ be a group and $H,K \in \mathfrak{C}(G)$, and consider $\mathfrak{F}^*$ and $U^*$ with respect to a fixed representative set $X$ with respect to $\sim$.
\begin{enumerate}
\item We have $\mathfrak{F}_H^*$ is closed under arbitrary unions and $\displaystyle \bigcup_{T \in \mathfrak{F}_H^*} T = U_H^*$.
\item We have $U_{H \vee K}^* = U_H^* \cap U_K^*$.  In other words, if we write $H = \mathbf{C}_G(U_H^*)$ and $K = \mathbf{C}_G(U_K^*)$, then $H \vee K = \mathbf{C}_G(U_H^* \cap U_K^*)$.
\end{enumerate}
\end{proposition}

\begin{proof}
For item (1), by Proposition \ref{prop: upper}, we have $H = \mathbf{C}_G(U_H^*)$, and $U_H^* \subseteq X$, so $U_H^* \in \mathfrak{F}_H^*$.  Thus $$U_H^* \subseteq \bigcup_{T \in \mathfrak{F}_H^*} T.$$ 

Now suppose that $\mathfrak{W} \subseteq \mathfrak{F}_H^*$ is arbitrary and let $\displaystyle U = \bigcup_{W \in \mathfrak{W}} W$.  By Proposition \ref{prop: cent_union_int_1}, $$ \mathbf{C}_G(U) = \mathbf{C}_G(\bigcup_{W \in \mathfrak{W}} W) =\bigcap_{W \in \mathfrak{W}} \mathbf{C}_G(W) = H,$$ and $U \subseteq X$, so $U \in \mathfrak{F}_H^*$, and $\mathfrak{F}_H^*$ is closed under arbitrary unions.  For each $u \in U$, $H = \mathbf{C}_G(U) \subseteq \mathbf{C}_G(u)$, and so $U \subseteq U_H^*$.  Taking $\mathfrak{W} = \mathfrak{F}_H^*$, we obtain $\displaystyle \bigcup_{T \in \mathfrak{F}_H^*} T \subseteq U_H^*$, and thus $\displaystyle \bigcup_{T \in \mathfrak{F}_H^*} T = U_H^*$.

For item (2), we have $\mathfrak{F}_{H \vee K}^* = \{ x \in X \, : \, H \vee K \subseteq \mathbf{C}_G(x) \}$. Note that $H \vee K \subseteq \mathbf{C}_G(x)$ if and only if $x \in \mathbf{C}_G(H \vee K) = \mathbf{C}_G(H) \cap \mathbf{C}_G(K)$.  Hence $$\mathfrak{F}_{H \vee K}^* = \{ x \in X \, : \, x \in \mathbf{C}_G(H) \,\, \text{and} \,\, x \in \mathbf{C}_G(K) \} = \mathfrak{F}_H^* \cap \mathfrak{F}_K^*.$$
\end{proof}

We now consider the case when $\langle S \rangle$ is abelian.

\begin{proposition} \label{prop: cent_ab}
Let $G$ be a group and $S \subseteq G$.  We have $\mathbf{C}_G(\mathbf{C}_G(S))$ is abelian if and only if $\langle S \rangle$ is abelian.
\end{proposition}

\begin{proof}    
We have $\langle S \rangle$ is abelian if and only if $S \subseteq \mathbf{C}_G(S)$ if and only if $\mathbf{C}_G(\mathbf{C}_G(S)) \subseteq \mathbf{C}_G(S)$ (by Proposition \ref{prop: cent_basic} item (2)).  And so $\langle S \rangle$ is abelian if and only if $$\mathbf{C}_G(\mathbf{C}_G(S)) \subseteq \mathbf{C}_G(S) = \mathbf{C}_G(\mathbf{C}_G(\mathbf{C}_G(S))),$$ i.e. if and only if $\mathbf{C}_G(\mathbf{C}_G(S))$ is abelian.
\end{proof}

We state the immediate corollary for element centralizers.

\begin{corollary}\label{cor: abelian_cent}
Let $G$ be a group and let $a \in G$.  Then $\mathbf{C}_G(\mathbf{C}_G(a))$ is abelian.
\end{corollary}

Since $\mathbf{C}_G(\mathbf{C}_G(a))$ is abelian, we have $$\mathbf{C}_G(\mathbf{C}_G(a)) \subseteq \mathbf{C}_G(\mathbf{C}_G(\mathbf{C}_G(a))) = \mathbf{C}_G(a),$$ and hence $\mathbf{C}_G(\mathbf{C}_G(a)) = \mathbf{Z}(\mathbf{C}_G(a))$.  We denote $\mathbf{C}_G(\mathbf{C}_G(a))$ by $\mathbf{Z}(a)$, and we refer to these subgroups as \textit{element centers}.  We write $\mathcal{Z}(G) = \{ \mathbf{Z}(g) \mid g \in G \setminus \mathbf{Z}(G) \}$ for the set of all non-central element centers and we write $\mathcal{C}(G) = \{ \mathbf{C}_G(g) \mid g \in G \setminus \mathbf{Z}(G) \}$ for the set of all proper element centralizers.

We know by Proposition \ref{prop: equiv_C-CC} that $\mathbf{C}_G(x) = \mathbf{C}_G(y)$ if and only if $\mathbf{C}_G(\mathbf{C}_G(x)) = \mathbf{C}_G(\mathbf{C}_G(y))$, i.e. if and only if $\mathbf{Z}(x) = \mathbf{Z}(y)$.  We write $\mathbf{Z}^*(g)$ for the equivalence class of $g$ under this relation.  So $$\mathbf{Z}^*(g) = \{ x \in G \,\, | \,\, \mathbf{Z}(x) = \mathbf{Z}(g) \} = \{ x \in G \,\, | \,\, \mathbf{C}_G(x) = \mathbf{C}_G(g) \}.$$

Note that $\mathbf{Z}^*(z) = \mathbf{Z}(G)$ for $z \in \mathbf{Z}(G)$. 

\begin{proposition} \label{prop: g_in_Z}
Let $G$ be a group and let $g \in G$.  We have $\mathbf{Z}^*(g) \subseteq \mathbf{Z}(g)$. 
\end{proposition}

\begin{proof}
If $x \in \mathbf{Z}^*(g)$, then $x \in \mathbf{Z}(x) = \mathbf{Z}(g)$.
\end{proof}

We now come to the main theorem of this section, which states that every centralizer can be written as the disjoint union of the equivalence classes of the element centers that it contains.  This generalizes a result in \cite{cover} that makes this observation in finite groups for only the element centers themselves.

\begin{theorem} \label{thm: z_stars}
Let $G$ be a group and let $H \in \mathfrak{C}(G)$.  We have $$ H = \bigcup_{\mathbf{Z}(x) \subseteq H} \mathbf{Z}^*(x) = \bigcup_{\mathbf{C}_G(H) \subseteq \mathbf{C}_G(x)} \mathbf{Z}^*(x),$$ and this union is disjoint over distinct $\mathbf{Z}(x)$ (or over distinct $\mathbf{C}_G(x)$).
\end{theorem}

\begin{proof}
That $\displaystyle \bigcup_{\mathbf{Z}(x) \subseteq H} \mathbf{Z}^*(x) = \bigcup_{\mathbf{C}_G(H) \subseteq \mathbf{C}_G(x)} \mathbf{Z}^*(x)$ follows by Corollary \ref{cor: bij}.  

By Proposition \ref{prop: g_in_Z}, each such $\mathbf{Z}^*(x) \subseteq \mathbf{Z}(x) \subseteq H$, and so $\displaystyle \bigcup_{\mathbf{Z}(x) \subseteq H} \mathbf{Z}^*(x) \subseteq H$.

Conversely, let $g \in H$.  So $\mathbf{C}_G(H) \subseteq \mathbf{C}_G(g)$, and $g \in \mathbf{Z}^*(g)$, and thus $$g \in \bigcup_{\mathbf{C}_G(H) \subseteq \mathbf{C}_G(x)} \mathbf{Z}^*(x).$$  Hence $\displaystyle H \subseteq \bigcup_{\mathbf{C}_G(H) \subseteq \mathbf{C}_G(x)} \mathbf{Z}^*(x)$, and we have equality.

Distinct $\mathbf{Z}(x)$'s (or distinct $\mathbf{C}_G(x)$'s) yield distinct $\mathbf{Z}^*(x)$'s, and as the $\mathbf{Z}^*(x)$'s are equivalence classes, a union over distinct equivalence classes is a disjoint union.
\end{proof}

We state the corollary for finite nonabelian groups.

\begin{corollary} \label{cor: z_stars_fin}
Let $G$ be a finite nonabelian group and let $H \in \mathfrak{C}(G) \setminus \{ \mathbf{Z}(G) \}$.  Let $g_1, \dots, g_t$ be elements of $G \setminus \mathbf{Z}(G)$ such that $\mathbf{Z}(g_1), \dots, \mathbf{Z}(g_t)$ are distinct and comprise all of the non-central element centers contained in $H$ (equivalently, $\mathbf{C}_G(g_1), \dots, \mathbf{C}_G(g_t)$ are distinct and comprise all of the proper element centralizers that contain $\mathbf{C}_G(H)$).  We have $\displaystyle H = (\bigcup_{i=1}^t \mathbf{Z}^*(g_i) )\cup \mathbf{Z}(G)$, and this union is disjoint.
\end{corollary}

\section{The M\"{o}bius function and an application to \texorpdfstring{$p$}{p}-groups}\label{sec: Moby}

In this section, we define the M\"obius function on a poset and apply it to the poset of element centers in a $p$-group.   We begin with an important lemma.

\begin{lemma}\label{lem: cen_divides_stars}
Let $G$ be a finite group.  For each $g \in G$, we can express $\mathbf{Z}^*(g)$ as a union of cosets of $\mathbf{Z}(G)$ in $G$.  Hence, for each $g \in G$, $|\mathbf{Z}(G)|$ divides $|\mathbf{Z}^*(g)|$, and $\displaystyle \frac{|\mathbf{Z}^*(g)|}{|\mathbf{Z}(G)|}$ is the number of elements of a transversal for $\mathbf{Z}(G)$ in $G$ that lie in $\mathbf{Z}^*(g)$.
\end{lemma}

\begin{proof}
Note that if $x \in \mathbf{Z}^*(g)$, then $x \mathbf{Z}(G) \subseteq \mathbf{Z}^*(g)$.  This is true because $\mathbf{C}_G(x) = \mathbf{C}_G(xz)$ for every $z \in \mathbf{Z}(G)$.  Thus, for each $g \in G$, we can express $\mathbf{Z}^*(g)$ as a union of cosets of $\mathbf{Z}(G)$ in $G$. 

The remaining claims follow easily as cosets are disjoint and each coset has size $|\mathbf{Z}(G)|$.
\end{proof}

Let $\mathcal{P}$ be a finite poset, and suppose $\mathcal{P}$ has a unique minimal element, say $\hat{0}$.  The \textit{M\"{o}bius function on $\mathcal{P}$}, $\mu : \mathcal{P} \rightarrow \mathbb{Z}$ (see \cite{stan}) can be defined recursively by

\[ \mu(x) = \begin{cases} 
      1 & \text{ if } x = \hat{0}, \\
      \displaystyle -\sum_{y < x} \mu(y) & \text{ if } x > \hat{0}. 
   \end{cases}
\]

We now apply the M\"obius function to $\mathcal{Z} (G) \cup \{ \mathbf{Z}(G) \}$ for $G$ a $p$-group and obtain our first main result of this section.
\begin{proposition} \label{prop: p-Mob}
Let $G$ be a $p$-group and let $\mathcal{Z}(G) = \{ \mathbf{Z}(g) \,\, | \,\, g \in G \setminus \mathbf{Z}(G) \}$.  Let $\mu$ be the M\"{o}bius function defined on the poset $\mathcal{Z}(G) \cup \{ \mathbf{Z}(G) \}$ under containment.  For each $g \in G$, we have $$\frac{|\mathbf{Z}^*(g)|}{|\mathbf{Z}(G)|} \equiv \mu(\mathbf{Z}(g)) \bmod p.$$ 
\end{proposition}

\begin{proof}
The proof is by induction on the number of element centers properly contained in an element center $\mathbf{Z}(g)$.  The base case is when $g \in \mathbf{Z}(G)$ and so $\mathbf{Z}^*(g) = \mathbf{Z}(G)$, and the result holds since $\mu(\mathbf{Z}(G)) = 1$.

Let $g \in G \setminus \mathbf{Z}(G)$.  Let $g_1, \dots, g_t$ be elements of $G \setminus \mathbf{Z}(G)$ such that $\mathbf{Z}(g_1), \dots, \mathbf{Z}(g_t)$ are distinct and comprise all of the element centers properly contained in $\mathbf{Z}(g)$, where $t \geq 0$.  By Corollary \ref{cor: z_stars_fin}, $$\mathbf{Z}(g) = \mathbf{Z}^*(g) \cup (\bigcup_{i=1}^t \mathbf{Z}^*(g_i) )\cup \mathbf{Z}(G),$$ and this union is disjoint.

By the induction hypothesis, for each $i$, $$k_i = \frac{ |\mathbf{Z}^*(g_i)|}{|\mathbf{Z}(G)|} \equiv \mu(\mathbf{Z}(g_i)) \bmod p.$$ Hence, $$|\mathbf{Z}^*(g)| = |\mathbf{Z}(g)| - \big(\sum_{i=1}^t |\mathbf{Z}(G)|k_i \big) - |\mathbf{Z}(G)|.$$

Now, $\mathbf{Z}(G)$ is a proper subgroup of $\mathbf{Z}(g)$, and we obtain 

$$\frac{|\mathbf{Z}^*(g)|}{|\mathbf{Z}(G)|} = |\mathbf{Z}(g) : \mathbf{Z}(G)| - \big(\sum_{i=1}^t k_i \big) - 1,$$ where $p$ divides $|\mathbf{Z}(g) : \mathbf{Z}(G)|$.

Hence, $$\frac{|\mathbf{Z}^*(g)|}{|\mathbf{Z}(G)|} \equiv - \big(\sum_{i=1}^t \mu(\mathbf{Z}(g_i))\big) - 1 \bmod p.$$

And $\displaystyle \mu(\mathbf{Z}(g)) = -\Big(\big(\sum_{i=1}^t \mu(\mathbf{Z}(g_i))\big) + 1\Big)$, since the $\mathbf{Z}(g_i)$'s and $\mathbf{Z}(G)$ are all of the distinct element centers that are properly contained in $\mathbf{Z}(g)$.  Hence $$\frac{|\mathbf{Z}^*(g)|}{|\mathbf{Z}(G)|} \equiv \mu(\mathbf{Z}(g)) \bmod p.$$ 
\end{proof}

We come to a main result in this section. View the poset of element centers in a group, $G$, as a subposet of the lattice of all centralizers, $\mathfrak{C}(G)$.  For a $p$-group, $G$, and for any centralizer $H \in \mathfrak{C}(G)$, we apply the M\"obius function, $\mu$, to the poset of element centers and we obtain a result on the sum of $\mu$ over all non-central element centers contained in $H$.

\begin{theorem} \label{thm: Mob}
 Let $G$ be a nonabelian $p$-group and let $\mathcal{Z}(G) = \{ \mathbf{Z}(g) \mid g \in G \setminus \mathbf{Z}(G) \}$ and let $\mathcal{C}(G) = \{ \mathbf{C}_G(g) \mid g \in G \setminus \mathbf{Z}(G) \}$.  Let $\mu$ be the M\"{o}bius function defined on the poset $\mathcal{Z}(G) \cup \{ \mathbf{Z}(G) \}$ under containment.  
 \begin{enumerate}
\item For each $H \in \mathfrak{C}(G)$ properly containing $\mathbf{Z}(G)$, we have $$\sum_{\substack{Z \in \mathcal{Z}(G), \\ Z \subseteq H}} \mu(Z) \equiv -1 \bmod p.$$  
\item For each $H \in \mathfrak{C}(G)$ that is properly contained in $G$, we have $$\sum_{\substack{C \in \mathcal{C}(G), \\ H \subseteq C}} \mu(\mathbf{Z}(C)) \equiv -1 \bmod p.$$
\end{enumerate}
\end{theorem}

\begin{proof}
We prove item (1), and item (2) follows by the duality of $\mathfrak{C}(G)$ (Corollary \ref{cor: bij}).

Let $g_1, \dots, g_t$ be elements of $G \setminus \mathbf{Z}(G)$ such that $\mathbf{Z}(g_1), \dots, \mathbf{Z}(g_t)$ are distinct and comprise all of the non-central element centers contained in $H$.

Since the corresponding $\mathbf{Z}^*$'s along with $\mathbf{Z}(G)$ form a partition of $H$ (Corollary \ref{cor: z_stars_fin}), we have $$|H| = \big( \sum_{i=1}^t |\mathbf{Z}^*(g_i)| \big) + |\mathbf{Z}(G)|,$$ and hence $$|H : \mathbf{Z}(G)| = \Big( \sum_{i=1}^t \frac{|\mathbf{Z}^*(g_i)|}{|\mathbf{Z}(G)|} \Big) + 1.$$ 

As $H$ properly contains $\mathbf{Z}(G)$, $p$ divides $|H : \mathbf{Z}(G)|$.  Thus, by Proposition \ref{prop: p-Mob}, we obtain $$\sum_{\substack{Z \in \mathcal{Z}(G), \\ Z \subseteq C}} \mu(Z) \equiv -1 \bmod p.$$
\end{proof}

We now consider $F$-groups. Expanding on the work of It\^o in \cite{ito}, Rebmann in \cite{reb} defines a group $G$ to be an {\it $F$-group} if for all $x,y \in G \setminus \mathbf{Z}(G)$ the condition $\mathbf{C}_G(x) \subseteq \mathbf{C}_G (y)$ implies $\mathbf{C}_G (x) = \mathbf{C}_G (y)$.  Equivalently, by duality, a group $G$ is an $F$-group if for all $x,y \in G \setminus \mathbf{Z}(G)$, $\mathbf{Z}(x) \subseteq \mathbf{Z}(y)$ implies $\mathbf{Z}(x) = \mathbf{Z}(y)$.

It may come as a surprise that not all groups are $F$-groups.  Consider, for example, a direct product $G$ of nonabelian groups $G_1,\dots,G_n$ with respective elements $g_i \in G_i \setminus \mathbf{Z}(G_i)$ for all $i$.  Then we have a chain of proper element centralizers $$\mathbf{C}_G(g_1) \supset \mathbf{C}_G(g_1 g_2) \dots \supset \mathbf{C}_G(g_1 \cdots g_n).$$

We mention that Theorem \ref{thm: Mob} generalizes a result in \cite{cover} regarding $F$-groups that are $p$-groups.  If $G$ is an $F$-group, then each $Z \in \mathcal{Z}(G)$ has $\mu(Z) = -1$, where $\mu$ is the M\"obius function defined as in Theorem \ref{thm: Mob}.  Taking $H = G$ in Theorem \ref{thm: Mob} (1), it follows that an $F$-group, $G$, that is a $p$-group has $|\mathcal{Z}(G)| \equiv 1 \bmod p$.  This result found in \cite{cover} is itself a generalization of a result of Verardi, see \cite{verardi}.  We add to this in the below corollary.

\begin{corollary}\label{cor: non_ab_F-Gp}
Suppose $G$ is a nonabelian $F$-group that is a $p$-group.  Let $\mathcal{Z}(G) = \{ \mathbf{Z}(g) \mid g \in G \setminus \mathbf{Z}(G) \}$ and let $\mathcal{C}(G) = \{ \mathbf{C}_G(g) \mid g \in G \setminus \mathbf{Z}(G) \}$.
\begin{enumerate}
\item If $H \in \mathcal{C}(G)$, then number of non-central element centers contained in $H$ is congruent to $1$ modulo $p$. 
\item If $H \in \mathcal{Z}(G)$, then the number of proper element centralizers that contain $H$ is congruent to $1$ modulo $p$. 
\end{enumerate}
\end{corollary}

\begin{proof}
Note that in an $F$-group, $\mu(Z) = -1$ for every non-central element center $Z$.  Applying Theorem \ref{thm: Mob} (1) to $H \in \mathcal{C}(G)$ and Theorem \ref{thm: Mob} (2) to $H \in \mathcal{Z}(G)$
the result follows.
\end{proof}

We pause for a moment to consider the relevance of our results on element centralizers and element centers to some graphs of groups related to the commuting graph. 

The commuting graph of a group $G$, which we will denote by $\mathfrak{G}(G)$, is a graph with vertex set $G \setminus \mathbf{Z}(G)$ and distinct elements $x,y$ have an edge between them if and only if $xy=yx$.  The commuting graph was probably introduced by Bertram (\cite{bertram}), and has been studied extensively.  Also of interest is the graph $\mathfrak{G}^*(G)$ which is the 
subgraph of $\mathfrak{G}(G)$ induced by a transversal for $\mathbf{Z}(G)$ in $G$.  This graph was introduced in $\cite{d2n}$.

In \cite{centgraph}, for a group $G$, we define the centralizer graph $\Gamma_{\mathcal{Z}} (G)$ to be the graph whose vertex set is $\mathcal{Z}(G) = \{ \mathbf{Z}(g) \,\, | \,\, g \in G \setminus \mathbf{Z}(G) \}$ and there is an edge between distinct $\mathbf{Z}(g)$ and $\mathbf{Z}(h)$ if $\mathbf{Z}(g) \subseteq \mathbf{C}_G (h)$.  This holds if and only if $\mathbf{Z}(h) \subseteq \mathbf{C}_G (g)$, by duality.  Hence this graph makes sense as an undirected graph, and we may also define it instead on the vertex set $\mathcal{C}(G) = \{ \mathbf{C}_G(g) \,\, | \,\, g \in G \setminus \mathbf{Z}(G) \}$.

In \cite{centgraph} we observe that the centralizer graph $\Gamma_{\mathcal{Z}}(G)$ is obtained from either the commuting graph $\mathfrak{C}(G)$ or the graph $\mathfrak{C}^*(G)$ by identifying vertices via the equivalence relation outlined in this paper: $x \sim y$ if and only if $\mathbf{C}_G(x) = \mathbf{C}_G(y)$.  

The following collects information on the neighbors and degrees of vertices in these three graphs.  We use the notation $N(x)$ for the neighbors of a vertex $x$ in a graph.

\begin{proposition}
Let $G$ be a finite group and let $\mathfrak{G}(G)$ be the commuting graph of $G$, $\mathfrak{G}^*(G)$ be the subgraph of $\mathfrak{G}(G)$ induced by a transversal $T$ for $\mathbf{Z}(G)$ over $G$, and $\Gamma_{\mathcal{Z}}(G)$ be the centralizer graph.  We have
\begin{enumerate}
\item If $x$ is a vertex of $\mathfrak{G}(G)$, then 
\[N(x) = (\mathbf{C}_G(x) \setminus \{x\}) \setminus \mathbf{Z}(G).\]
Thus, the degree of $x$ in $\mathfrak{G} (G)$ is $|{\bf C}_G (x)| - |\mathbf{Z}(G)| - 1$.
\item If $x$ is a vertex of $\mathfrak{G}^*(G)$, then 
\[N(x) = (\mathbf{C}_G(x) \cap T) \setminus (\{x\} \cup (\mathbf{Z}(G) \cap T)).\]
Thus, the degree of $x$ in $\mathfrak{G}^* (G)$ is $|\mathbf{C}_G (x):\mathbf{Z}(G)| - 2$.
\item If $\mathbf{Z}(x)$ is a vertex of $\Gamma_{\mathcal{Z}}(G)$, then 
\[N(\mathbf{Z}(x)) = \{ \mathbf{Z}(y) \in \mathcal{Z}(G) \mid \mathbf{Z}(y) \subseteq \mathbf{C}_G(x) \} \setminus \{ \mathbf{Z}(x) \}.\]
\end{enumerate}
\end{proposition}

If $\mathbf{Z}(G)$ is a nontrivial $p$-group, then the degree of every vertex in $\mathfrak{G}(G)$ is congruent to $-1$ modulo $p$ since both $|{\bf C}_G (x)|$ and $|\mathbf{Z}(G)|$ will be divisible by $p$.  When $G/\mathbf{Z}(G)$ is a nontrivial $p$-group, the degree of every vertex in $\mathfrak{G}^*(G)$ is congruent to $-2$ modulo $p$.  For the centralizer graph, $\Gamma_{\mathcal{Z}}(G)$, it is not true in general that a nonabelian $p$-group $G$ will have the degrees of its vertices all congruent to the same value modulo $p$.  One can find many examples, and in fact SmallGroup$(3^7,261)$, which we mention at the end of this paper, is one example. In the case of an $F$-group, however, we do obtain an analogous result for the centralizer graph.

\begin{corollary}\label{cor: cent_graph_degree}
Suppose $G$ is a nonabelian $F$-group that is a $p$-group.  Every vertex of the centralizer graph $\Gamma_{\mathcal{Z}}(G)$ has degree congruent to $0$ modulo $p$.
\end{corollary}

\begin{proof}
The result follows from Corollary \ref{cor: non_ab_F-Gp} (1) since the degree of vertex $\mathbf{Z}(x) \in \mathcal{Z}(G)$ is the number of non-central element centers contained in $\mathbf{C}_G(x)$ other than $\mathbf{Z}(x)$.
\end{proof}

We close by showing in Figure \ref{fig: Z-poset} the M\"obius function values on the poset $\mathcal{Z}(G) \cup \{ \mathbf{Z}(G) \}$ for a $3$-group, $G$, that is not an $F$-group.  One can verify that Theorem \ref{thm: Mob} holds in this example.  Take, for instance, $H = G$ in item (1).  The non-central element centers contained in $G$ would consist of all of the non-central element centers.  The sum of all of those M\"obius function values is $5$, which is $-1$ modulo $3$.  For another example, let $H$ be one of the element centers with $\mu(H) = 3$.  Adding up all of the M\"obius values for the non-central element centers contained in $H$ (which includes $H$ itself) yields $-1$, which is again $-1$ modulo $3$.

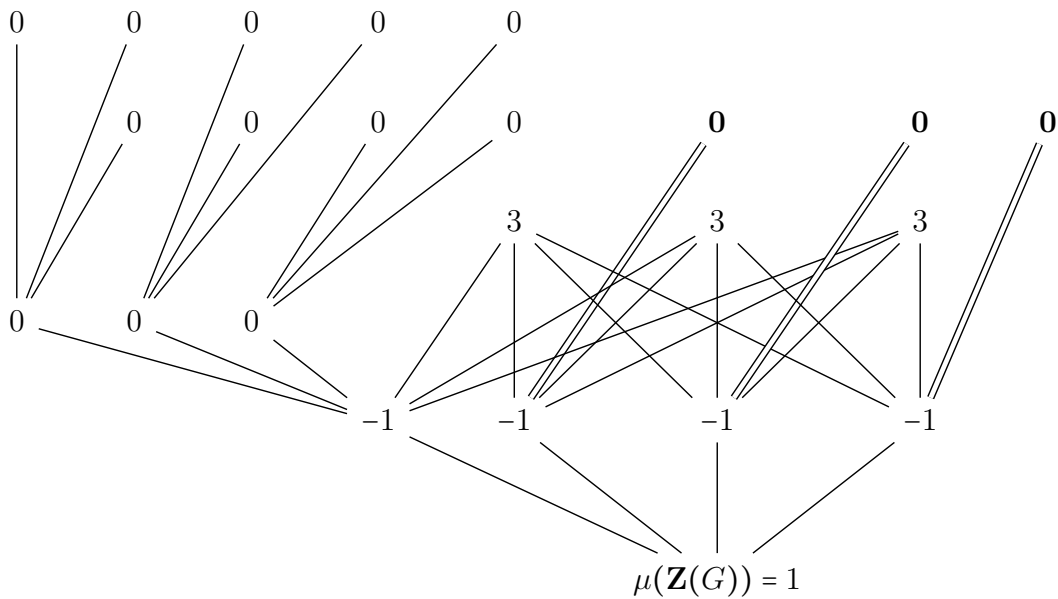
\begin{figure}[H]
\[\begin{tikzcd}
	0 & 0 & 0 & 0 & 0 \\
	& 0 & 0 & 0 & 0 & {\textbf{0}} & {\textbf{0}} & {\textbf{0}} \\
	&&&& {3} & {3} & {3} \\
	0 & 0 & 0 \\
	&&& {-1} & {-1} & {-1} & {-1} \\
	\\
	&&&&& {\mu(\mathbf{Z}(G))=1}
	\arrow[no head, from=1-1, to=4-1]
	\arrow[no head, from=1-2, to=4-1]
	\arrow[no head, from=1-3, to=4-2]
	\arrow[no head, from=1-4, to=4-2]
	\arrow[no head, from=1-5, to=4-3]
	\arrow[no head, from=2-2, to=4-1]
	\arrow[no head, from=2-3, to=4-2]
	\arrow[no head, from=2-4, to=4-3]
	\arrow[no head, from=2-5, to=4-3]
	\arrow[Rightarrow, no head, from=2-6, to=5-5]
	\arrow[no head, from=3-5, to=5-4]
	\arrow[no head, from=3-5, to=5-5]
	\arrow[no head, from=3-5, to=5-7]
	\arrow[no head, from=3-6, to=5-4]
	\arrow[no head, from=3-6, to=5-5]
	\arrow[no head, from=3-6, to=5-6]
	\arrow[no head, from=3-7, to=5-4]
	\arrow[no head, from=3-7, to=5-6]
	\arrow[no head, from=3-7, to=5-7]
	\arrow[no head, from=4-1, to=5-4]
	\arrow[no head, from=4-2, to=5-4]
	\arrow[no head, from=4-3, to=5-4]
	\arrow[no head, from=5-4, to=7-6]
	\arrow[no head, from=5-5, to=3-7]
	\arrow[no head, from=5-5, to=7-6]
	\arrow[Rightarrow, no head, from=5-6, to=2-7]
	\arrow[no head, from=5-6, to=3-5]
	\arrow[no head, from=5-6, to=7-6]
	\arrow[Rightarrow, no head, from=5-7, to=2-8]
	\arrow[no head, from=5-7, to=3-6]
	\arrow[no head, from=5-7, to=7-6]
\end{tikzcd}\]
\caption{The M\"obius function, $\mu$, defined on the poset $\mathcal{Z}(G) \cup \{ \mathbf{Z}(G) \}$ for a group $G$ isomorphic to SmallGroup$(3^7,261)$.  The bold-faced $\textbf{0}$ nodes each represent $27$ distinct nodes containing the node below.  Note that the sum of $\mu$ over the $100$ total non-minimal nodes is $5$, which is $-1$ modulo $3$.}
\label{fig: Z-poset}
\end{figure}

\noindent
{\bf Data Availability:}

There is no data associated with this paper.

\end{document}